\documentclass[10pt,conference]{ieeeconf}
\IEEEoverridecommandlockouts

\usepackage{amsmath,amssymb,amsfonts}
\usepackage{bm}
\usepackage{cite}
\usepackage{graphicx}
\usepackage{booktabs}
\let\labelindent\relax
\usepackage{enumitem}
\usepackage{balance}
\usepackage{xcolor}
\usepackage{amsmath,amssymb,amsfonts}
\usepackage{bm}
\usepackage{cite}
\usepackage{graphicx}
\usepackage{booktabs}
\usepackage{enumitem}
\usepackage{balance}
\usepackage{xcolor}
\newtheorem{theorem}{Theorem}
\newtheorem{proposition}{Proposition}
\newtheorem{remark}{Remark}
\newtheorem{assumption}{Assumption}

\newcommand{\R}{\mathbb{R}}
\newcommand{\E}{\mathbb{E}}
\newcommand{\cA}{\mathcal{A}}
\newcommand{\cC}{\mathcal{C}}
\newcommand{\cS}{\mathcal{S}}
\newcommand{\cN}{\mathcal{N}}

\newcommand{\norm}[1]{\left\lVert #1\right\rVert}
\DeclareMathOperator{\col}{col}
\DeclareMathOperator{\blkdiag}{blkdiag}

\newcommand{\vx}{\mathbf{x}}
\newcommand{\vy}{\mathbf{y}}
\newcommand{\vw}{\mathbf{w}}
\newcommand{\vv}{\mathbf{v}}
\newcommand{\vxi}{\bm{\xi}}

\newcommand{\vc}{\mathbf{c}}
\newcommand{\vd}{\mathbf{d}}

\newcommand{\vz}{\mathbf{z}}

\newcommand{\veta}{\bm{\eta}}
\newcommand{\mA}{\mathbf{A}}

\newcommand{\mH}{\mathbf{H}}
\newcommand{\mR}{\mathbf{R}}
\newcommand{\mW}{\mathbf{W}}
\newcommand{\mS}{\mathbf{S}}
\newcommand{\mK}{\mathbf{K}}

\newcommand{\mM}{\mathbf{M}}

\newcommand{\mF}{\mathbf{F}}
\newcommand{\mG}{\mathbf{G}}

\newcommand{\mI}{\mathbf{I}}
\newcommand{\mPhi}{\mathbf{\Phi}}
\newcommand{\mSigma}{\bm{\Sigma}}

\title{Source-Reliability Weighted Observer Design for Open Client–Server Networks}

\author{Amit Dutta, Marcos M. Vasconcelos  and Olugbenga M. Anubi  
\thanks{A. Dutta, M. M. Vasconcelos and O. M. Anubi are with the Department of Electrical and Computer Engineering, FAMU-FSU College of Engineering, Florida State University, Tallahassee, FL 32306, USA. E-mails:
        {\tt  \{adutta, m.vasconcelos, oanubi\}@fsu.edu}. }%
}

\begin{document}
\maketitle

\begin{abstract}
State estimation in open client–server networks is challenging because the set of active observation sources changes over time, and active sources need not be valid sensors of the latent state. We study a client-server estimation problem in which a fixed but unknown state-consistent sensing class measures the latent state when active, while nuisance sources may arrive in the active set and generate observations from a different signal class. The server observes active sources, but it does not know which agents observe valid state measurements. We propose a source-reliability-weighted observer. The state estimate is a standard fixed-prior weighted least-squares update, but the information assigned to each active observation is determined by a source-level reliability score learned from repeated innovation consistency. Numerical results show that the observer, when all-active sources are considered, has a finite bias, while the proposed observer initially learns source reliability and then tracks the desired latent state after a finite learning period.
\end{abstract}

\section{Introduction}
Networked sensing and control systems increasingly rely on measurements transmitted by sources whose participation changes over time because of mobility, packet drops, intermittent connectivity, or changing client availability \cite{sinopoli2004,schenato2007,ganti2011}. In a \textit{closed} and trusted sensor setup, each received observation is assumed to follow the prescribed measurement model for the latent state. In an \textit{open} client--server network, however, this assumption can fail. A server may receive measurements from sources that are active but not state-consistent: faulty sensors, mismatched environmental probes, malicious clients, or sources observing a different process. The resulting difficulty is not only intermittent observation, but uncertainty about which active source identities should influence the state estimate.

This paper studies the setting illustrated in Fig.~\ref{fig:problem_overview}. A latent state $\vx_t$ evolves according to a known linear model. A fixed but hidden subset $\cC\subset\cN$ of source identities forms the \textit{state-consistent} sensing class. When a source in $\cC$ is active, it provides a noisy measurement of $\vx_t$. A \textit{nuisance} source in $\cS=\cN\setminus\cC$ may also be active, but its observation is not assumed to be generated by the latent state. The server observes source identities and active observations, but it does not know the hidden class label ``state-consistent'' or ``nuisance.'' 
The target is to recover the latent state $\vx_t$, not the average of active observations and not the nuisance process.

\begin{figure}[t]
    \centering
    \includegraphics[width=1\linewidth]{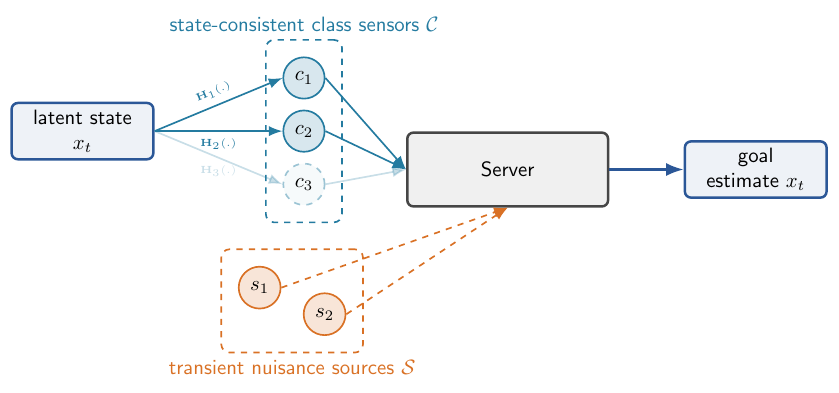}
    \caption{Latent state estimation in an open client--server network. State-consistent sources $\cC$ measure the latent state when active, while nuisance sources $\cS$ transmit observations that are not generated by the latent state model. The server observes only active sources and does not know the hidden source class.}
    \label{fig:problem_overview}
\end{figure}
The estimation layer uses an online prediction--correction least-squares update. At each time, it forms a model-based prediction and then applies a weighted correction using only the currently active observations. Unlike Kalman filtering and related recursive or moving-horizon least-squares methods, it does not solve a batch least-squares problem over a moving finite horizon or maintain the Riccati covariance recursion of a fully trusted linear-Gaussian filtering model.


This modeling choice is driven by the open-network setting, where the main uncertainty is not only measurement noise but also whether an active source is state-consistent or nuisance. The observer therefore uses the model prediction as a stabilizing anchor, corrects it with the currently active measurements, and learns over time which source identities should be trusted in the correction step.


The main contribution is a resilience mechanism for source-class separation under nuisance contamination. If the state-consistent source set $\cC$ were known, an oracle observer that uses only active sources in $\cC$ could track the latent state under the finite-window information and stability conditions developed later. In the open-network setting, however, $\cC$ is unknown. The proposed method therefore aims to recover the behavior of this oracle-valid observer by learning source reliability from repeated prediction residuals and using the learned reliability as a measurement-information weight in the correction step.

\subsection{Related work}
The estimator that we consider in this work is obtained from the same least-squares/Gaussian maximum a posteriori principle that underlies Kalman filtering, information filtering, and recursive least-squares estimation \cite{kalman1960,anderson1979,kailath2000,jazwinski1970,sayed2003}. The distinction is that the covariance matrix is not updated by a Riccati recursion; it is a fixed design matrix that determines how strongly the correction remains anchored to the model prediction.

The work also differs from robust filtering and robust statistics. Classical robust estimators usually reduce the effect of an individual measurement whose current residual is large, through residual reweighting or saturated losses or likelihood assignment \cite{huber1981,masreliez1977, schick1994}. Such methods are often observation-level: a measurement is downweighted because its current residual is large. In contrast, the proposed method assigns trust to source identities. A valid source can produce a large residual because of noise or prediction error, and a nuisance source can produce a small residual by chance. The observer therefore learns a reliability score for each source identity from repeated prediction-residual behavior and uses that score to weight the source's future measurement information.

The problem is also distinct from blind source separation and classical data association. Blind source separation and independent component analysis aim to recover latent components or unmixing transformations from mixtures, typically through independence, non-Gaussianity, or mixing assumptions \cite{comon1994,hyvarinen2000}. Data association in tracking addresses the assignment of anonymous measurements to targets in clutter \cite{barshalom1975,barshalom2009}. Here, the measurements are already indexed by source identity and the measurement matrices are known. The hidden variable is the validity of each source identity.

This work is also related to resilient estimation and learning under faulty or adversarial agents. Attack-resilient state estimation typically asks when the true state can be recovered despite a bounded number of compromised sensors, using conditions such as secure observability, sensor redundancy, or adversary-detection mechanisms \cite{pajic2014,pajic2017,chen2018resilient}. Resilient distributed optimization and learning similarly rely on redundancy across local objectives or gradients to tolerate Byzantine failures or open participation \cite{GuptaVaidya2020,GuptaDoanVaidya2021,LiuGuptaVaidya2023,DuttaDoan2025OpenRedundancy,DuttaDoanReed2023ByzantineFL}. These redundancy-based formulations are powerful but restrictive for open sensing networks: a nuisance source need not be a corrupted copy of a valid source or an adversarial perturbation of the same measurement task; it may correspond to a different signal class altogether. The present paper therefore does not assume source-level redundancy. Instead, it requires that the state-consistent sources, once identified with sufficiently high reliability, provide finite-window observability of the latent state.

\subsection{Contribution}

This paper formulates the problem as source-validity learning for an online reliability-weighted estimator. The problem follows a standard weighted least-squares form; the distinction lies in how measurement information is assigned in an open network where active source identities are observed but their validity is unknown. At each time, the estimator anchors the update at the model-based state prediction, while a reliability variable learns from repeated prediction residuals, which source identities are consistent with the latent-state measurement model. This separates two issues that are often conflated: intermittent participation and source validity. A source may be active only occasionally and still be state-consistent, while another may appear frequently but generate nuisance measurements. Once the learned weights preserve information from valid sources and suppress nuisance leakage, the resulting time-varying error recursion satisfies a post-learning tracking bound. The simulations illustrate this mechanism: the all-active observer remains biased by nuisance signals, whereas the proposed estimator separates the source classes and recovers tracking behavior close to the state-consistent oracle.

\section{Problem Setup and Fixed-Prior Observer Derivation}
\subsection{Open client--server sensing model}
Let $\cN=\{1,\ldots,N\}$ denote the set of source identities observed over the horizon. A fixed but unknown subset $\cC\subset\cN$ contains the state-consistent sensors. The complement $\cS=\cN\setminus\cC$ contains nuisance or non-state-consistent sources. The latent state evolves as
\begin{equation}
    \vx_{t+1}=\mA\vx_t+\vw_t,\qquad \vx_t\in\R^d,
    \label{eq:state}
\end{equation}
where $\mA\in\R^{d\times d}$ is known and $\vw_t$ is the process disturbance. For each source $i\in\cN$, let $s_i(t)\in\{0,1\}$ denote its activity indicator and define
\begin{equation}
    \cA_t=\{i\in\cN:s_i(t)=1\}.
\end{equation}
If state-consistent source $i\in\cC$ is active, the server receives
\begin{equation}
    \vy_i(t)=\mH_i\vx_t+\vv_i(t),\qquad i\in\cC,
    \label{eq:state_consistent_meas}
\end{equation}
where $\vy_i(t)\in\R^{m_i}$, $\mH_i\in\R^{m_i\times d}$ is known, and $\vv_i(t)$ is measurement noise. If nuisance source $j\in\cS$ is active, the server receives
\begin{equation}
    \vy_j(t)=\vxi_j(t),\qquad j\in\cS,
    \label{eq:nuisance_signal}
\end{equation}
where $\vxi_j(t)$ is an unknown nuisance signal. We do not assume that $\vxi_j(t)$ has a known distribution or is generated by the latent state model.

The server knows $\mA$ and the source-indexed matrices $\{\mH_i,\mR_i\}_{i\in\cN}$, where $\mR_i\succ0$ is the nominal covariance or information scaling assigned to a valid measurement from source $i$. It observes $\{(i,\vy_i(t)):i\in\cA_t\}$ but does not know $\cC$. The target is to estimate/track $\vx_t$.


\subsection{Fixed-prior least-squares observer viewpoint}
Before introducing the all-active, oracle, and proposed observers, we define the common prediction--correction layer used by all three. The prediction step is
\begin{equation}
    \hat{\vx}_{t|t-1}=\mA\hat{\vx}_{t-1|t-1}.
    \label{eq:prediction}
\end{equation}

The prediction $\hat{\vx}_{t|t-1}$ is treated as a prior pseudo-measurement of $\vx_t$ through the identity map. The positive definite matrix $\mSigma$ is a fixed prediction-confidence covariance: its inverse is the information assigned to the model prediction. For a valid source $i$, the matrix $\mR_i\succ0$ is the nominal measurement-noise covariance or information scaling associated with the residual $\vy_i(t)-\mH_i\vx$; hence $\mR_i^{-1}$ weights that residual in information form. With these interpretations, the prior and measurement fitting terms are
\begin{equation}
    \norm{\vx-\hat{\vx}_{t|t-1}}_{\mSigma^{-1}}^2,
    \qquad
    \norm{\vy_i(t)-\mH_i\vx}_{\mR_i^{-1}}^2 .
\end{equation}
If the prediction error and valid measurement noise were Gaussian with these fixed covariances, the resulting quadratic correction would coincide with a one-step maximum-a-posteriori/least-squares estimate \cite{kalman1960,jazwinski1970,kailath2000}. In this paper the same quadratic form is used as an online resilient estimation layer, without assuming that every active measurement is valid.

The update is not a finite-horizon least-squares estimator since rather than using measurements over a past time window, our estimator uses only the one-step model prediction and the currently active measurements, not a past measurement window. It is also not a full Kalman filter since the correction uses a fixed prediction-confidence matrix rather than a Riccati covariance recursion \cite{kalman1960,anderson1979,kailath2000}. Although recursive least squares also updates estimates from prior information and measurement residuals \cite{sayed2003}, our setting keeps $\mSigma$ fixed and focuses on source validity: $\mSigma$ regularizes each instantaneous correction, while the learned source weights determine which active identities contribute measurement information.

\subsection{Prediction confidence and one-step anchoring}
The covariance-like matrices $\mSigma$ and $\mR_i$ have different roles. The matrix $\mSigma$ controls trust in the model prediction, while $\mR_i$ controls the nominal scale of source-$i$ measurement noise if that source is state-consistent. A small $\mSigma$ strongly anchors the correction to the prediction, while a large $\mSigma$ makes the update more measurement-responsive. In the scalar one-measurement case with $H=1$, $R=\sigma_v^2$, and $\Sigma=\sigma_x^2$, the one-step correction is
\begin{equation}
    \hat{x}_{t|t}=\hat{x}_{t|t-1}+g(y_t-\hat{x}_{t|t-1}),
    \qquad
    g=\frac{\sigma_x^2}{\sigma_x^2+\sigma_v^2}.
    \label{eq:scalar_prior_example}
\end{equation}
Thus $\sigma_x^2$ directly sets the fraction of the prediction residual injected into the observer. This scalar calculation is only an intuition: in the multidimensional open-network case, the correction direction and strength depend on the active measurement geometry and learned weights. This is why the later simulation checks finite-window observability and product stability rather than relying on scalar intuition.

\subsection{Prediction anchoring and local learning regime}


We assume that the underlying valid-source estimation problem is solvable: if nuisance sources were absent, or if the state-consistent source class were known, the estimator using only valid active measurements would track the latent state under the finite-window information and stability conditions stated below. The purpose of reliability learning is to recover this valid-source behavior when nuisance sources are present and the valid source class is unknown.

\begin{assumption}[Local prediction accuracy during reliability learning]
\label{ass:local_prediction_accuracy}
There exist constants $r_{\rm learn}>0$ and $T_{\rm sep}<\infty$ such that, over the learning interval $0\le t<T_{\rm sep}$, the prediction error satisfies
\begin{equation}
    \norm{\hat{\vx}_{t|t-1}-\vx_t}\le r_{\rm learn}.
    \label{eq:local_prediction_accuracy}
\end{equation}
\end{assumption}

Assumption~\ref{ass:local_prediction_accuracy} is used only to justify the source-separation mechanism. It states that the server starts from an informative estimate, and that the process and measurement corruptions during the learning phase do not drive the predictor so far away that residual-based reliability becomes meaningless. After the source weights have separated, the tracking result below is governed by the post-learning error recursion and its product stability.

\subsection{All-active and oracle observers}
With the notation above, the naive all-active observer is
\begin{equation}
\begin{aligned}
\hat{\vx}_{t|t}^{\rm all}=\arg\min_{\vx}\Big\{&
\norm{\vx-\hat{\vx}_{t|t-1}^{\rm all}}_{\mSigma^{-1}}^2 \\
&+\sum_{i\in\cA_t}\norm{\vy_i(t)-\mH_i\vx}_{\mR_i^{-1}}^2\Big\}.
\end{aligned}
\label{eq:all_active}
\end{equation}
This estimator is well-defined, but it is generally not correct when $\cA_t$ contains nuisance sources because terms of the form $\mH_j^\top\mR_j^{-1}\vxi_j(t)$ enter the normal equations even though $\vxi_j(t)$ is not generated by $\vx_t$.

The unattainable oracle observer uses only active state-consistent sources:
\begin{equation}
\begin{aligned}
\hat{\vx}_{t|t}^{\rm or}=\arg\min_{\vx}\Big\{&
\norm{\vx-\hat{\vx}_{t|t-1}^{\rm or}}_{\mSigma^{-1}}^2 \\
&+\sum_{i\in\cC\cap\cA_t}\norm{\vy_i(t)-\mH_i\vx}_{\mR_i^{-1}}^2\Big\}.
\end{aligned}
\label{eq:oracle_cost}
\end{equation}
The oracle is not implementable because $\cC$ is unknown. It is used as a diagnostic reference for the proposed reliability-weighted observer. In the absence of nuisance sources, or if the state-consistent class were known, this fixed-prior observer reduces to an online prediction--correction estimator driven by valid measurements. Under the finite-window observability and post-learning product-stability conditions stated later, the induced error recursion contracts. The resilience objective is to recover this valid-source observer behavior when nuisance identities are present and unknown.

\begin{remark}[Why all-active estimation is biased]
In the scalar static case, if $n_{\rm c}$ valid measurements satisfy $y_i=x+v_i$ with $\E[v_i]=0$ and $n_{\rm s}$ nuisance measurements have mean $\mu_{\rm s}$, then the all-active average obeys
\begin{equation}
    \E[\hat{x}^{\rm all}-x]
    =\frac{n_{\rm s}}{n_{\rm c}+n_{\rm s}}(\mu_{\rm s}-x).
    \label{eq:mixture_bias}
\end{equation}
The bias is structural: it remains even when the valid measurement noises are zero mean. This motivates identity-level reliability weights instead of treating every active observation as valid.
\end{remark}

\subsection{Source-reliability weighted observer}
The proposed observer attaches a reliability weight $\omega_i(t)\in[0,1]$ to each source. At time $t$, it solves
\begin{equation}
\begin{aligned}
\hat{\vx}_{t|t}=\arg\min_{\vx\in\R^d}\Big\{&
\norm{\vx-\hat{\vx}_{t|t-1}}_{\mSigma^{-1}}^2 \\
&+\sum_{i\in\cA_t}\omega_i(t)
\norm{\vy_i(t)-\mH_i\vx}_{\mR_i^{-1}}^2\Big\}.
\end{aligned}
\label{eq:weighted_cost}
\end{equation}
A small $\omega_i(t)$ inflates the effective covariance of source $i$; if $\omega_i(t)=0$, the observation is ignored. Thus the reliability weight is not an additional residual penalty outside the observer. It is the measurement-information multiplier inside the weighted least-squares correction.

Let $\cA_t=\{i_1,\ldots,i_{a_t}\}$ and $m_t^{\rm a}=\sum_{\ell=1}^{a_t}m_{i_\ell}$. Define
\begin{align}
\vy_t&:=\col\{\vy_{i_1}(t),\ldots,\vy_{i_{a_t}}(t)\}\in\R^{m_t^{\rm a}},
\label{eq:Ystack}\\
\mH_t&:=\col\{\mH_{i_1},\ldots,\mH_{i_{a_t}}\}\in\R^{m_t^{\rm a}\times d},
\label{eq:Hstack}\\
\mW_t^\omega&:=\blkdiag\{\omega_{i_1}(t)\mR_{i_1}^{-1},\ldots,
\omega_{i_{a_t}}(t)\mR_{i_{a_t}}^{-1}\}.
\label{eq:Wstack}
\end{align}
Then \eqref{eq:weighted_cost} becomes
\begin{equation}
\hat{\vx}_{t|t}=\arg\min_{\vx}
\left\{
\norm{\vx-\hat{\vx}_{t|t-1}}_{\mSigma^{-1}}^2
+\norm{\vy_t-\mH_t\vx}_{\mW_t^\omega}^2
\right\}.
\label{eq:stacked_cost}
\end{equation}
If $\cA_t=\emptyset$, the measurement term is absent and the update keeps the prediction.

\begin{proposition}[Closed form, well-posedness, and recursive innovation form]
Let $\mSigma\succ0$, $\mR_i\succ0$, and $\omega_i(t)\ge0$. Define
\begin{align}
\mM_t^\omega
&:=\mH_t^\top \mW_t^\omega\mH_t
=\sum_{i\in\cA_t}\omega_i(t)\mH_i^\top\mR_i^{-1}\mH_i,
\label{eq:Mtomega}\\
\mS_t^\omega
&:=\mSigma^{-1}+\mM_t^\omega,
\label{eq:Stomega}\\
\vc_t^\omega
&:=\mSigma^{-1}\hat{\vx}_{t|t-1}+\mH_t^\top\mW_t^\omega\vy_t,
\label{eq:ctomega}\\
\mK_t^\omega
&:=(\mS_t^\omega)^{-1}\mH_t^\top\mW_t^\omega.
\label{eq:Komega}
\end{align}
Then \eqref{eq:weighted_cost} has the unique minimizer
\begin{equation}
    \hat{\vx}_{t|t}=(\mS_t^\omega)^{-1}\vc_t^\omega.
    \label{eq:closed}
\end{equation}
Moreover, the same update can be written in recursive innovation form as
\begin{equation}
    \hat{\vx}_{t|t}=\hat{\vx}_{t|t-1}+\mK_t^\omega
    \left(\vy_t-\mH_t\hat{\vx}_{t|t-1}\right).
    \label{eq:recursive_update}
\end{equation}
\end{proposition}
\begin{proof}
The gradient of \eqref{eq:stacked_cost} is
\begin{equation}
2\mSigma^{-1}(\vx-\hat{\vx}_{t|t-1})-2\mH_t^\top\mW_t^\omega(\vy_t-\mH_t\vx).
\end{equation}
Setting it to zero gives $\mS_t^\omega\vx=\vc_t^\omega$. Since $\mSigma^{-1}\succ0$ and $\mW_t^\omega\succeq0$, $\mS_t^\omega\succ0$, so the minimizer is unique and equals \eqref{eq:closed}. Adding and subtracting $\mH_t^\top\mW_t^\omega\mH_t\hat{\vx}_{t|t-1}$ in $\vc_t^\omega$ gives the innovation form \eqref{eq:recursive_update}. 
\end{proof}

\begin{remark}[Well-posedness versus observability]
The condition $\mSigma\succ0$ guarantees that the weighted least-squares problem has a unique minimizer at each time, even when few sources are active or when the instantaneous active measurement matrix is rank deficient. This does not guarantee that the active valid measurements observe the state over time. Per-step well-posedness is a regularization property; convergence is a dynamical observability and stability property.
\end{remark}
We further have the following
\begin{align}\label{eq:IKH_identity}
    \mI-\mK_t^\omega\mH_t
=\mI-(\mS_t^\omega)^{-1}\mM_t^\omega
&=(\mS_t^\omega)^{-1}(\mS_t^\omega-\mM_t^\omega)\notag\\
&=(\mS_t^\omega)^{-1}\mSigma^{-1}.
\end{align}
\subsection{Scalar reliability-weighted form}
For intuition only, take $d=1$, $H_i=1$, $R_i=\sigma_v^2$, and $\Sigma=\sigma_x^2$. If $\Omega_t=\sum_{i\in\cA_t}\omega_i(t)$, then
\begin{equation}
\hat{x}_{t|t}=\hat{x}_{t|t-1}+g_t^\omega(\bar y_t^\omega-\hat{x}_{t|t-1}),
\label{eq:scalar_closed_weighted}
\end{equation}
where $g_t^\omega=\sigma_x^2\Omega_t/(\sigma_v^2+\sigma_x^2\Omega_t)$. Hence both $\Sigma$ and the learned weights control correction strength. The scalar case is relatively simple because every valid measurement provides information about the same one-dimensional state. In the multidimensional case, it is not enough to identify reliable sources; the reliable active measurements must collectively observe every state direction over time.

\section{Source-Reliability Weighted Algorithm}
In the proposed method, the weights $\omega_i(t)$ are not known a priori; they are learned online from repeated prediction-residual evidence associated with each source identity. This is necessary because a single observation is not sufficient to classify a source: a state-consistent sensor can temporarily produce a large prediction residual because of noise or prediction error, while a nuisance source can occasionally appear consistent by chance. Thus, the weighted least-squares observer is the estimation layer, and the reliability-learning mechanism determines how much measurement information each active source contributes over time.

\subsection{Prediction-residual consistency score and residual temperature}
Given the prediction $\hat{\vx}_{t|t-1}$, define the normalized prediction residual for every active source by
\begin{equation}
    r_i(t)=\norm{\vy_i(t)-\mH_i\hat{\vx}_{t|t-1}}_{\mR_i^{-1}}^2.
    \label{eq:innovation}
\end{equation}
This residual is small when the observation is compatible with the predicted state model and large when it is incompatible. We map it to a bounded consistency score
\begin{equation}
    \zeta_i(t)=\phi(r_i(t))\in[0,1],
    \label{eq:score}
\end{equation}
where $\phi$ is nonincreasing. In the simulations we use
\begin{equation}
    \zeta_i(t)=\exp\left(-\frac{r_i(t)}{\tau}\right),
    \label{eq:zeta}
\end{equation}
where $\tau>0$ is a residual temperature. Small $\tau$ makes the test strict, so moderately large residuals receive small scores. Large $\tau$ makes the test tolerant, so scores decay slowly with residual size.

The local prediction condition in Assumption~\ref{ass:local_prediction_accuracy} explains why the residual-derived score can carry source-class information. When the predictor is reasonably close to the true state, a valid source produces a prediction residual dominated by measurement noise and prediction error. A nuisance source produces a prediction residual containing a structural mismatch because its signal is not generated by $\mH_i\vx_t$. Hence the expected consistency score of a valid source can exceed that of a nuisance source.

\subsection{Reliability memory and mean-score interpretation}
Let
\begin{equation}
    N_i(t):=\sum_{k=0}^{t}\mathbf{1}\{i\in\cA_k\}
\end{equation}
be the number of times source $i$ has appeared. A natural empirical-average reliability score is
\begin{equation}
q_i(t)=q_i(t-1)+\frac{1}{N_i(t)}\bigl(\zeta_i(t)-q_i(t-1)\bigr),\qquad i\in\cA_t,
\label{eq:empirical_q}
\end{equation}
with inactive scores held fixed. This update makes $q_i(t)$ a source-level memory of repeated consistency scores. In the simulations we use the constant-gain version
\begin{equation}
q_i(t)=
\begin{cases}
(1-\beta)q_i(t-1)+\beta\zeta_i(t), & i\in\cA_t,\\
q_i(t-1), & i\notin\cA_t,
\end{cases}
\label{eq:q_update}
\end{equation}
where $\beta\in(0,1]$ controls the memory speed.

For interpretation, let $n$ count appearances of one representative source and define $\bar q_{\rm c}(n)$ and $\bar q_{\rm s}(n)$ as the expected scores of representative valid and nuisance sources after $n$ appearances. If
\begin{equation}
\mu_{\rm c}:=\E[\zeta_i\mid i\in\cC],\qquad
\mu_{\rm s}:=\E[\zeta_j\mid j\in\cS],
\label{eq:mean_consistency_scores}
\end{equation}
then the constant-gain recursion gives $\bar q_{\rm c}(n+1)=(1-\beta)\bar q_{\rm c}(n)+\beta\mu_{\rm c}$ and $\bar q_{\rm s}(n+1)=(1-\beta)\bar q_{\rm s}(n)+\beta\mu_{\rm s}$. With equal initialization,
\begin{equation}
\bar q_{\rm c}(n)-\bar q_{\rm s}(n)
=\left(1-(1-\beta)^n\right)(\mu_{\rm c}-\mu_{\rm s}).
\label{eq:mean_score_gap_solution}
\end{equation}
Thus, when $\mu_{\rm c}>\mu_{\rm s}$, repeated residual-score memory creates an expected score gap. The phrase expected prediction-residual consistency refers precisely to these class-conditional expectations of the bounded score $\zeta_i(t)$; it does not introduce an additional signal beyond the residual score. A finite-time sample-path separation theorem is left for future work.

\subsection{Score-to-weight gate and post-learning separation}
The observer weight is obtained by applying a soft gate to the score:
\begin{equation}
    \omega_i(t)=\sigma_\kappa(q_i(t)-\lambda)
    :=\frac{1}{1+\exp[-\kappa(q_i(t)-\lambda)]},
    \label{eq:soft_gate}
\end{equation}
where $\lambda\in(0,1)$ is the transition threshold and $\kappa>0$ is the gate sharpness. The score $q_i(t)$ is the memory variable; the weight $\omega_i(t)$ is the measurement-information multiplier used in the observer.

The reliability-learning mechanism and the observer analysis address two distinct issues: source-class separation and state-estimation convergence. Here, we focus on the latter. Therefore, the convergence analysis is carried out under a post-learning separation condition, which assumes that after some finite time state-consistent sources receive sufficiently large weights and nuisance sources receive sufficiently small weights. A finite-time probabilistic guarantee for this separation is left for future work.

\begin{assumption}[Post-learning reliability separation]
\label{ass:reliability_separation}
There exist $T_{\rm sep}<\infty$, $\underline{\omega}_{\cC}>0$, and $\bar{\omega}_{\cS}\ge0$, with $\bar{\omega}_{\cS}<\underline{\omega}_{\cC}$, such that for all $t\ge T_{\rm sep}$,
\begin{equation}
\omega_i(t)\ge \underline{\omega}_{\cC},\quad i\in\cC\cap\cA_t,
\qquad
\omega_j(t)\le \bar{\omega}_{\cS},\quad j\in\cS\cap\cA_t .
\label{eq:post_learning_weight_separation}
\end{equation}
\end{assumption}
\begin{remark}
Assumption~\ref{ass:reliability_separation} is a post-learning condition, not a statement that the weights separate abruptly at $T_{\rm sep}$. The reliability scores are updated continuously from observed prediction residuals, so the separation may develop gradually. The time $T_{\rm sep}$ denotes the point after which valid active sources are persistently assigned larger weights than nuisance active sources.
\end{remark}

\section{Error Dynamics and Post-Learning Convergence}
This section derives the estimation-error recursion and a post-learning tracking bound. No independent gain design is performed: the correction matrix is induced by the fixed-prior WLS problem, the active set, and the learned source weights.

\subsection{Estimation-error recursion}
Define
\begin{equation}
    \tilde{\vx}_{t|t}:=\hat{\vx}_{t|t}-\vx_t,
    \qquad
    \tilde{\vx}_{t|t-1}:=\hat{\vx}_{t|t-1}-\vx_t .
\end{equation}
Using \eqref{eq:state} and \eqref{eq:prediction},
\begin{align}
\tilde{\vx}_{t|t-1}
&=\mA\hat{\vx}_{t-1|t-1}-(\mA\vx_{t-1}+\vw_{t-1}) \notag\\
&=\mA\tilde{\vx}_{t-1|t-1}-\vw_{t-1}.
\label{eq:prediction_error_revised}
\end{align}
For the stacked active observation, write
\begin{equation}
    \vy_t=\mH_t\vx_t+\vd_t,
    \label{eq:stacked_disturbance_model_revised}
\end{equation}
where $\vd_t$ stacks the valid measurement noises for $i\in\cC\cap\cA_t$ and the nuisance mismatch terms $\vxi_j(t)-\mH_j\vx_t$ for $j\in\cS\cap\cA_t$. Substituting \eqref{eq:stacked_disturbance_model_revised} into \eqref{eq:recursive_update} gives
\begin{align}
\tilde{\vx}_{t|t}
&=\tilde{\vx}_{t|t-1}+\mK_t^\omega(\mH_t\vx_t+\vd_t-\mH_t\hat{\vx}_{t|t-1})\notag\\
&=(\mI-\mK_t^\omega\mH_t)\tilde{\vx}_{t|t-1}+\mK_t^\omega\vd_t.
\end{align}
Combining this with \eqref{eq:prediction_error_revised} yields
\begin{equation}
\tilde{\vx}_{t|t}
=\mF_t^\omega\tilde{\vx}_{t-1|t-1}+\veta_t,
\label{eq:error_recursion_compact_revised}
\end{equation}
where
\begin{align}
\mF_t^\omega&:=(\mI-\mK_t^\omega\mH_t)\mA,
\label{eq:Ft_definition_revised}\\
\veta_t&:=-(\mI-\mK_t^\omega\mH_t)\vw_{t-1}+\mK_t^\omega\vd_t.
\label{eq:error_input_definition_revised}
\end{align}
The matrix $\mF_t^\omega$ is the induced observer error matrix. It depends on the current active set and reliability weights. Therefore, the post-learning stability condition must involve products of these matrices, not the eigenvalues of one fixed matrix.
\label{sec:simulations}
\begin{figure}[t]
\centering
\includegraphics[width=1\linewidth]{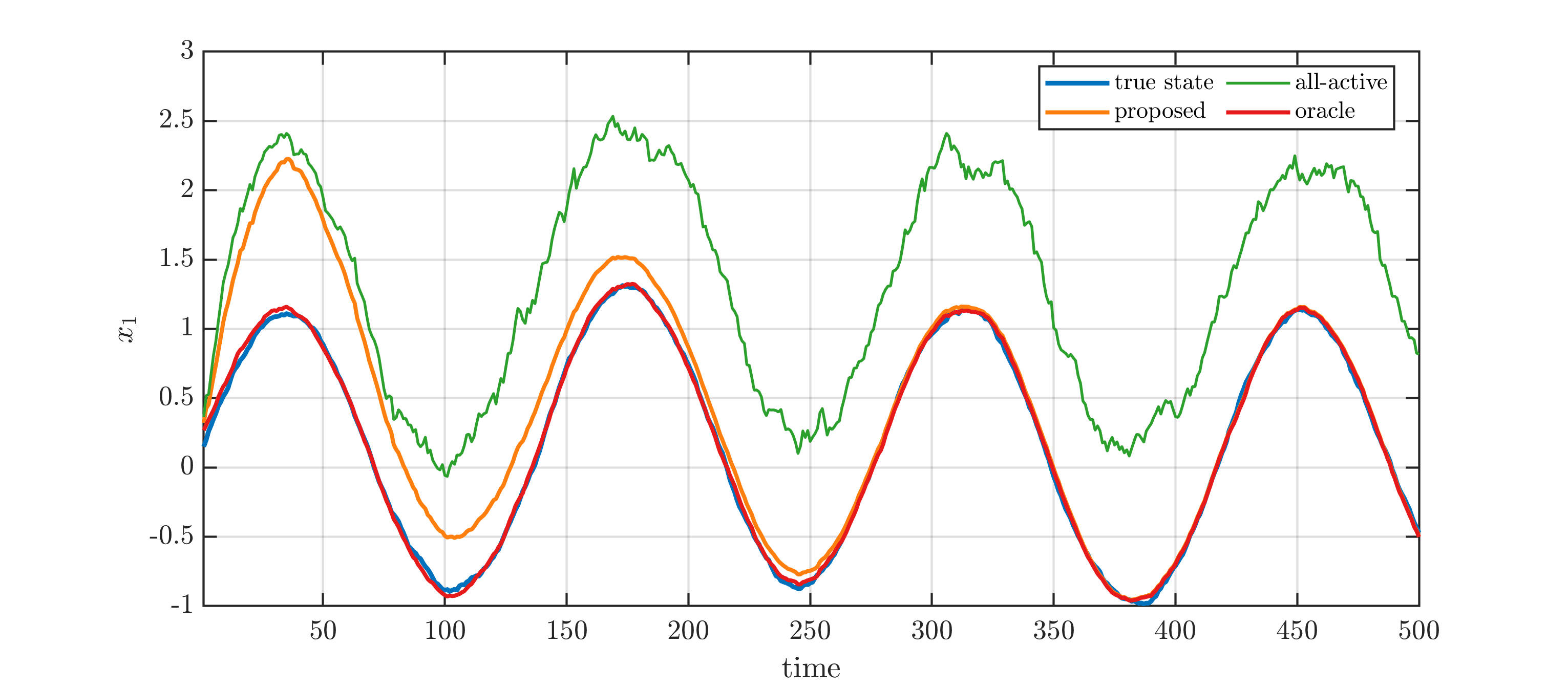}
\caption{State-estimation comparison. The naive all-active observer has a
finite bias because it uses nuisance observations. The proposed observer is
initially imperfect but improves after reliability learning and then tracks the
latent state.}
\label{fig:sim_tracking}
\end{figure}
Using \eqref{eq:IKH_identity}, the same matrix can be written as
\begin{equation}
    \mF_t^\omega=(\mS_t^\omega)^{-1}\mSigma^{-1}\mA
    =(\mI+\mSigma\mM_t^\omega)^{-1}\mA.
    \label{eq:stability_object_revised}
\end{equation}
This identity is useful for interpretation: measurement information in $\mM_t^\omega$ damps the model propagation $\mA$, while missing or weakly trusted measurements leave the observer closer to open-loop prediction.

\subsection{Finite-window state-consistent observability}
The observer can recover the state only if the active valid sources observe all state directions over time. For a window length $L\ge1$, define
\begin{equation}
\mG_{t,L}^{\cC,\omega}
:=
\sum_{\ell=t}^{t+L-1}
(\mA^{\ell-t})^\top
\left(
\sum_{i\in\cC\cap\cA_\ell}
\omega_i(\ell)\mH_i^\top\mR_i^{-1}\mH_i
\right)
\mA^{\ell-t}.
\label{eq:weighted_observability_gramian_revised}
\end{equation}
For any direction $\vz\in\R^d$,
\begin{equation}
\vz^\top\mG_{t,L}^{\cC,\omega}\vz
=\sum_{\ell=t}^{t+L-1}\sum_{i\in\cC\cap\cA_\ell}
\omega_i(\ell)\norm{\mH_i\mA^{\ell-t}\vz}_{\mR_i^{-1}}^2.
\label{eq:gramian_direction_energy_revised}
\end{equation}
Thus $\mG_{t,L}^{\cC,\omega}$ measures the total reliability-weighted valid measurement energy in each state direction over the window. We refer to $\mG_{t,L}^{\cC,\omega}$ as a reliability-weighted finite-window observability Gramian because it is a positive-semidefinite information matrix accumulated from valid measurements over a finite window.

\begin{assumption}[Finite-window state-consistent observability]
\label{ass:finite_window_observability}
There exist $L\ge1$ and $\alpha>0$ such that, for every post-learning window,
\begin{equation}
    \mG_{t,L}^{\cC,\omega}\succeq \alpha\mI.
    \label{eq:finite_window_observability}
\end{equation}
\end{assumption}
\begin{figure*}[t]
\centering
\includegraphics[width=1\linewidth]{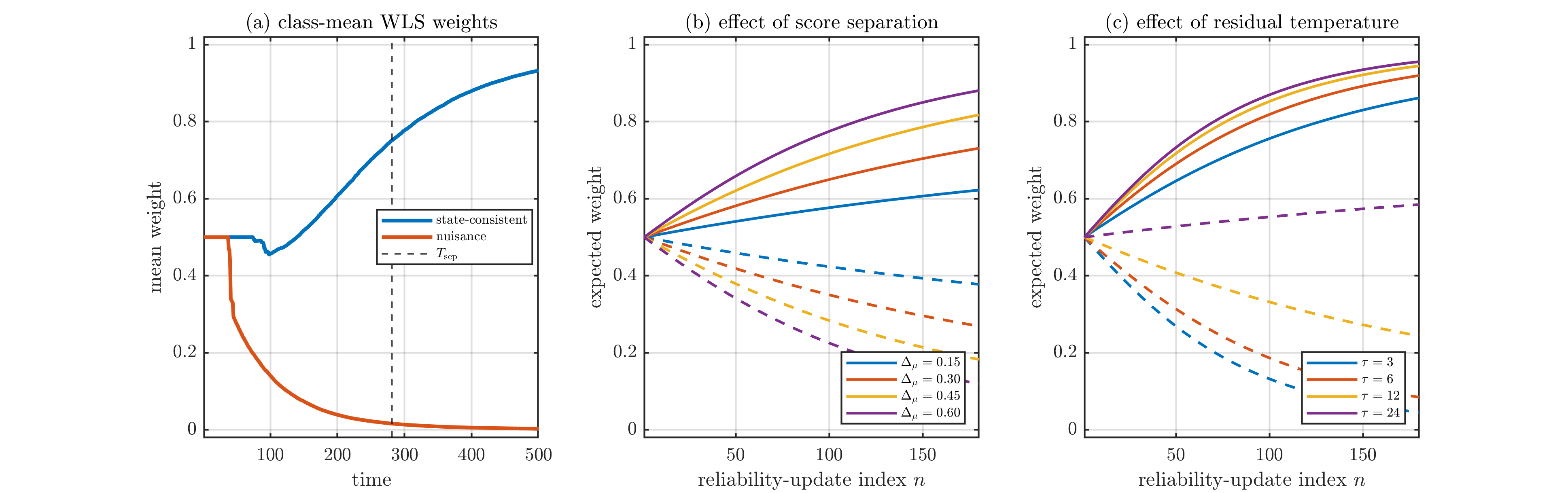}
\caption{Reliability-weight behavior. Panel (a) shows class-mean weights from the full simulation: state-consistent source weights increase while nuisance-source weights decrease. Panels (b)--(c) are mean-score diagnostics indexed by the reliability-update counter $n$, not by global time. Larger score separation $\Delta_\mu=\mu_{\rm c}-\mu_{\rm s}$ gives faster source discrimination, while the residual-temperature parameter $\tau$ controls how residual mismatch is translated into consistency scores. Solid curves denote state-consistent weights and dashed curves denote nuisance weights.}
\label{fig:sim_reliability}
\end{figure*}
Assumption~\ref{ass:finite_window_observability} is a tracking requirement, not a per-step solvability requirement. The matrix $\mSigma^{-1}$ makes each instantaneous WLS problem uniquely solvable, but it does not ensure that the valid measurement stream observes every state direction.

\subsection{Post-learning product stability}
The product term comes directly from rolling out \eqref{eq:error_recursion_compact_revised}. Unlike $\mG_{t,L}^{\cC,\omega}$, the matrix $\mPhi_\omega(t,s)$ defined below is not a Gramian; it is the state-transition product of the time-varying homogeneous error recursion. For $t=s$,
\begin{equation}
\tilde{\vx}_{s|s}=\mF_s^\omega\tilde{\vx}_{s-1|s-1}+\veta_s.
\end{equation}
For $t=s+1$,
\begin{equation}
\tilde{\vx}_{s+1|s+1}=\mF_{s+1}^\omega\mF_s^\omega\tilde{\vx}_{s-1|s-1}
+\mF_{s+1}^\omega\veta_s+\veta_{s+1}.\notag
\end{equation}
For general $t\ge s$, define
\begin{equation}
    \mPhi_\omega(t,s):=\mF_t^\omega\mF_{t-1}^\omega\cdots\mF_s^\omega,
    \label{eq:transition_product_revised}
\end{equation}
and use the empty-product convention $\mPhi_\omega(t,t+1)=\mI$ only to write the current forcing term compactly. The rolled-out recursion is
\begin{equation}
\tilde{\vx}_{t|t}=\mPhi_\omega(t,s)\tilde{\vx}_{s-1|s-1}
+\sum_{\ell=s}^{t}\mPhi_\omega(t,\ell+1)\veta_\ell.
\label{eq:error_rollout_revised}
\end{equation}

\begin{assumption}[Post-learning product stability]
\label{ass:product_stability}
There exist constants $c_\omega\ge1$ and $\rho_\omega\in(0,1)$ such that, for all $t\ge s\ge T_{\rm sep}$,
\begin{equation}
    \norm{\mPhi_\omega(t,s)}_2
    \le c_\omega\rho_\omega^{t-s+1}.
    \label{eq:product_stability_bound}
\end{equation}
\end{assumption}

In the scalar case, this assumption reduces to the decay of products of scalars $F_t^\omega$. In the multidimensional case, it can be checked numerically by computing $\norm{\mPhi_\omega(t,s)}_2$ over post-learning windows. The simulation below performs this check and also verifies Assumption~\ref{ass:finite_window_observability} empirically.

\subsection{Convergence theorem}
\begin{theorem}[Post-learning tracking bound]
\label{thm:post_learning_error_bound}
Suppose Assumptions~\ref{ass:reliability_separation} and~\ref{ass:product_stability} hold. Then, for all $t\ge s\ge T_{\rm sep}$,
\begin{equation}
\norm{\tilde{\vx}_{t|t}}
\le
c_\omega\rho_\omega^{t-s+1}\norm{\tilde{\vx}_{s-1|s-1}}
+c_\omega\sum_{\ell=s}^{t}\rho_\omega^{t-\ell}\norm{\veta_\ell}.
\label{eq:post_learning_error_bound}
\end{equation}
\end{theorem}

\begin{proof}
Rolling out \eqref{eq:error_recursion_compact_revised} from $s$ to $t$ gives \eqref{eq:error_rollout_revised}. Taking norms and using the triangle inequality,
\[
\norm{\tilde{\vx}_{t|t}}
\le \norm{\mPhi_\omega(t,s)}_2\norm{\tilde{\vx}_{s-1|s-1}}
+\sum_{\ell=s}^{t}\norm{\mPhi_\omega(t,\ell+1)}_2\norm{\veta_\ell}.
\]
Assumption~\ref{ass:product_stability} bounds the first product by $c_\omega\rho_\omega^{t-s+1}$ and the product multiplying $\veta_\ell$ by $c_\omega\rho_\omega^{t-\ell}$. For $\ell=t$, the empty product is $\mI$, which is also bounded by $c_\omega\rho_\omega^0$ because $c_\omega\ge1$. Substitution yields \eqref{eq:post_learning_error_bound}.
\end{proof}

\begin{remark}[Implications of the post-learning error bound]
\label{rem:post_learning_error_implications}
Theorem~\ref{thm:post_learning_error_bound} separates the initial post-learning estimation error from the aggregate forcing term $\veta_t$. If $\veta_t\equiv0$ after $T_{\rm sep}$, then $\tilde{\vx}_{t|t}\to0$ exponentially. If $\norm{\veta_t}\le \bar\eta$, then
\begin{equation}
    \limsup_{t\to\infty}\norm{\tilde{\vx}_{t|t}}
    \leq
    \frac{c_\omega}{1-\rho_\omega}\bar\eta .
    \label{eq:ultimate_error_bound}
\end{equation}
Thus, with bounded process noise, measurement noise, or residual nuisance leakage, the observer is ultimately bounded rather than exactly convergent. If $\veta_t\to0$, then the stable convolution term vanishes and $\tilde{\vx}_{t|t}\to0$. In this paper, this conclusion is established under Assumption~\ref{ass:product_stability}, which imposes the post-learning product-stability condition needed for the error recursion. The simulations provide empirical evidence that this condition holds. 

\end{remark}

\begin{remark}[Post-learning separation and future work]
\label{rem:post_learning_future_work}
Assumption~\ref{ass:reliability_separation} isolates the observer-convergence question from the reliability-separation question. After $T_{\rm sep}$, valid active sources have weights at least $\underline{\omega}_{\cC}$, so their measurement information is preserved, while nuisance sources have weights at most $\bar{\omega}_{\cS}$, so their mismatch contribution to $\mK_t^\omega\vd_t$ is suppressed. A finite-time probabilistic theorem proving such separation and quantifying $T_{\rm sep}$ is left for future work.
\end{remark}

\section{Simulations}
\label{sec:simulations}

We simulate a three-dimensional open client--server sensing problem with $15$ source identities. We consider $9$ sources that are state-consistent and $6$ that are nuisance sources. The latent state is generated by
\begin{equation}
\vx_{t+1}=\mA\vx_t+\vw_t ,
\label{eq:sim_dynamics}
\end{equation}
where
\begin{align*}
    \mA=
\begin{bmatrix}
0.998 & 0.008 & 0\\
-0.004 & 0.996 & 0.006\\
0 & -0.005 & 0.995
\end{bmatrix},
\end{align*}
and $\vw_t$ is zero-mean process noise with covariance $0.009^2\mI$. The desired nominal trajectory has three different sinusoidal coordinates, with offsets $(0.10,-0.20,0.18)$, amplitudes $(1.05,0.78,0.93)$, frequencies $(0.045,0.031,0.058)$, and phases $(0,1.10,2.20)$. The state-consistent sources generate measurements according to the latent-state measurement model, with scalar measurement noise variance $0.0925^2$. Their measurement rows are repeated coordinate measurements, so the valid source class observes all three state directions over finite windows. The nuisance sources generate structured oscillatory signals with offsets and phases that are not produced by the latent-state measurement model. Thus, the setup matches the open-network problem: source identities and measurements are observed, but the server does not know which active sources are state-consistent.

We compare three estimators. The all-active estimator assigns unit weight to every active source and therefore treats nuisance measurements as valid. The oracle estimator uses only the active state-consistent sources and serves as a benchmark. The proposed estimator uses the reliability-weighted least-squares correction with fixed prediction-confidence matrix $(0.025)^2\mI$ and learns source weights from repeated prediction residuals. The empirical separation time $T_{\rm sep}$ is defined as the first time after which the mean state-consistent weight remains above the chosen high threshold and the mean nuisance weight remains below the chosen low threshold for all later samples.

Fig. \ref{fig:sim_tracking} plot shows the first representative coordinate. The proposed estimator is initially affected by nuisance contamination, but after the source weights separate, it recovers the latent-state trajectory and tracks close to the oracle.

Fig. \ref{fig:sim_reliability} (a) shows the mean observer weights for the two source classes. We observe that, after the learning transient, state-consistent sources are assigned large correction weights while nuisance sources are assigned small weights, so the observer update increasingly uses measurements generated by the latent-state model and suppresses measurements generated by the nuisance signal class. Fig. \ref{fig:sim_reliability} (b) shows that larger score separation gives faster class discrimination, while \ref{fig:sim_reliability} (c) shows that the residual-temperature parameter controls how strongly residual mismatch affects the learned weights.

Fig \ref{fig:sim_leakage} (a) shows that finite-window observability Gramian remains positive, indicating that the retained valid measurements observe all state directions, thus adhering to Assumption \ref{ass:finite_window_observability}. The post-learning transition product decays as shown in Fig \ref{fig:sim_leakage} (b), illustrating the product-stability condition in Assumption \ref{ass:product_stability}. Finally, the nuisance-leakage plot in Fig \ref{fig:sim_leakage} (c), shows the additive effect of nuisance measurements in the error recursion. This leakage is non-negligible before source separation and drops after nuisance weights are suppressed. 
\begin{figure*}[t]
\centering
\includegraphics[width=1\linewidth]{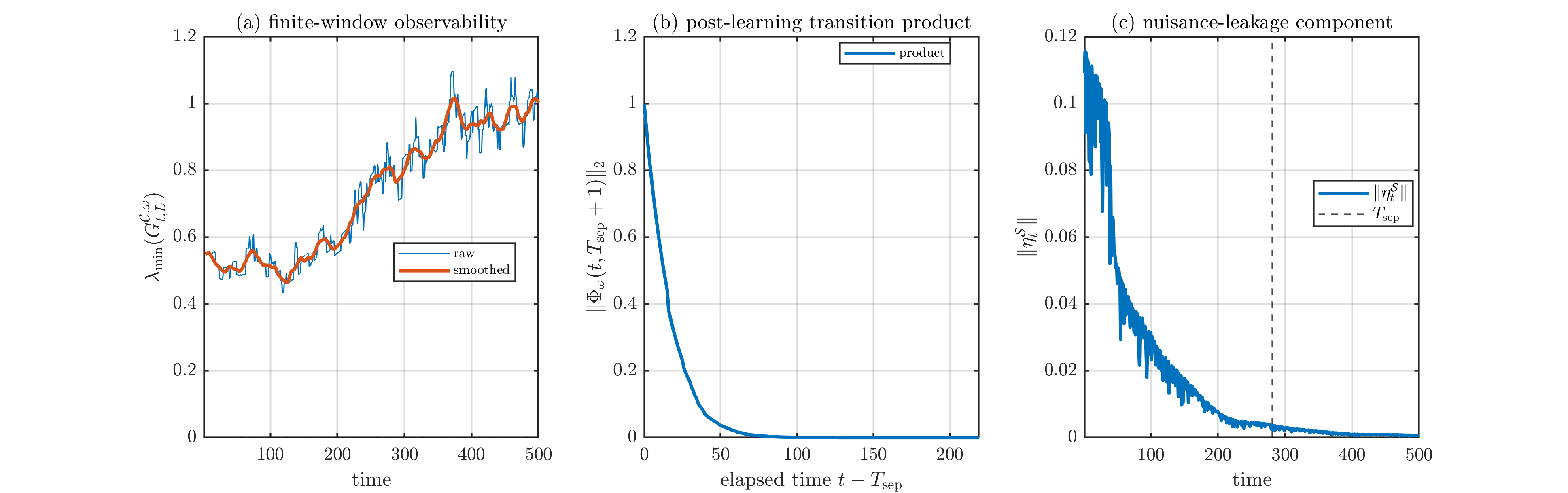}
\caption{Nuisance-leakage component of the error forcing. Before separation, nuisance sources still contribute to the WLS correction and produce a non-negligible additive forcing term. After separation, the learned nuisance weights are small and the leakage drops to a lower noise floor.}
\label{fig:sim_leakage}
\end{figure*}

\section{Conclusion}
We presented a source-reliability weighted observer for state estimation in an open client-server network with nuisance contamination. The observer combines a standard fixed-prior weighted least-squares correction with source-level reliability learning based on innovation consistency. We derived the closed-form update, the recursive innovation form, and the score-to-weight mechanism. Numerical results show the intended behavior: the proposed observer initially suffers from imperfect reliability, then separates state-consistent and nuisance sources, and finally tracks the latent state while the all-active observer remains biased.

\bibliographystyle{IEEEtran}
\bibliography{IEEEfull,myBIB}

\begin{thebibliography}{10}
\providecommand{\url}[1]{#1}
\csname url@samestyle\endcsname
\providecommand{\newblock}{\relax}
\providecommand{\bibinfo}[2]{#2}
\providecommand{\BIBentrySTDinterwordspacing}{\spaceskip=0pt\relax}
\providecommand{\BIBentryALTinterwordstretchfactor}{4}
\providecommand{\BIBentryALTinterwordspacing}{\spaceskip=\fontdimen2\font plus
\BIBentryALTinterwordstretchfactor\fontdimen3\font minus \fontdimen4\font\relax}
\providecommand{\BIBforeignlanguage}[2]{{%
\expandafter\ifx\csname l@#1\endcsname\relax
\typeout{** WARNING: IEEEtran.bst: No hyphenation pattern has been}%
\typeout{** loaded for the language `#1'. Using the pattern for}%
\typeout{** the default language instead.}%
\else
\language=\csname l@#1\endcsname
\fi
#2}}
\providecommand{\BIBdecl}{\relax}
\BIBdecl

\bibitem{sinopoli2004}
B.~Sinopoli, L.~Schenato, M.~Franceschetti, K.~Poolla, M.~I. Jordan, and S.~S. Sastry, ``Kalman filtering with intermittent observations,'' \emph{IEEE Transactions on Automatic Control}, vol.~49, no.~9, pp. 1453--1464, 2004.

\bibitem{schenato2007}
L.~Schenato, B.~Sinopoli, M.~Franceschetti, K.~Poolla, and S.~S. Sastry, ``Foundations of control and estimation over lossy networks,'' \emph{Proceedings of the IEEE}, vol.~95, no.~1, pp. 163--187, 2007.

\bibitem{ganti2011}
R.~K. Ganti, F.~Ye, and H.~Lei, ``Mobile crowdsensing: Current state and future challenges,'' \emph{IEEE Communications Magazine}, vol.~49, no.~11, pp. 32--39, 2011.

\bibitem{kalman1960}
R.~E. Kalman, ``A new approach to linear filtering and prediction problems,'' \emph{Journal of Basic Engineering}, vol.~82, no.~1, pp. 35--45, 1960.

\bibitem{anderson1979}
B.~D.~O. Anderson and J.~B. Moore, \emph{Optimal Filtering}.\hskip 1em plus 0.5em minus 0.4em\relax Englewood Cliffs, NJ, USA: Prentice-Hall, 1979.

\bibitem{kailath2000}
T.~Kailath, A.~H. Sayed, and B.~Hassibi, \emph{Linear Estimation}.\hskip 1em plus 0.5em minus 0.4em\relax Upper Saddle River, NJ, USA: Prentice Hall, 2000.

\bibitem{jazwinski1970}
A.~H. Jazwinski, \emph{Stochastic Processes and Filtering Theory}.\hskip 1em plus 0.5em minus 0.4em\relax Academic Press, 1970.

\bibitem{sayed2003}
A.~H. Sayed, \emph{Fundamentals of Adaptive Filtering}.\hskip 1em plus 0.5em minus 0.4em\relax Wiley, 2003.

\bibitem{huber1981}
P.~J. Huber, \emph{Robust Statistics}.\hskip 1em plus 0.5em minus 0.4em\relax New York, NY, USA: John Wiley \& Sons, 1981.

\bibitem{masreliez1977}
C.~J. Masreliez and R.~D. Martin, ``Robust bayesian estimation for the linear model and robustifying the kalman filter,'' \emph{IEEE Transactions on Automatic Control}, vol.~22, no.~3, pp. 361--371, 1977.

\bibitem{schick1994}
I.~C. Schick and S.~K. Mitter, ``Robust recursive estimation in the presence of heavy-tailed observation noise,'' \emph{The Annals of Statistics}, vol.~22, no.~2, pp. 1045--1080, 1994.

\bibitem{comon1994}
P.~Comon, ``Independent component analysis, a new concept?'' \emph{Signal Processing}, vol.~36, no.~3, pp. 287--314, 1994.

\bibitem{hyvarinen2000}
A.~Hyv{\"a}rinen and E.~Oja, ``Independent component analysis: Algorithms and applications,'' \emph{Neural Networks}, vol.~13, no. 4--5, pp. 411--430, 2000.

\bibitem{barshalom1975}
Y.~Bar-Shalom and E.~Tse, ``Tracking in a cluttered environment with probabilistic data association,'' \emph{Automatica}, vol.~11, no.~5, pp. 451--460, 1975.

\bibitem{barshalom2009}
Y.~Bar-Shalom, F.~Daum, and J.~Huang, ``The probabilistic data association filter,'' \emph{IEEE Control Systems Magazine}, vol.~29, no.~6, pp. 82--100, 2009.

\bibitem{pajic2014}
M.~Pajic, J.~Weimer, N.~Bezzo, P.~Tabuada, O.~Sokolsky, I.~Lee, and G.~J. Pappas, ``Robustness of attack-resilient state estimators,'' in \emph{Proceedings of the ACM/IEEE International Conference on Cyber-Physical Systems}, 2014, pp. 163--174.

\bibitem{pajic2017}
M.~Pajic, I.~Lee, and G.~J. Pappas, ``Attack-resilient state estimation for noisy dynamical systems,'' \emph{IEEE Transactions on Control of Network Systems}, vol.~4, no.~1, pp. 82--92, 2017.

\bibitem{chen2018resilient}
Y.~Chen, S.~Kar, and J.~M.~F. Moura, ``Resilient distributed estimation through adversary detection,'' \emph{IEEE Transactions on Signal Processing}, vol.~66, no.~9, pp. 2455--2469, 2018.

\bibitem{GuptaVaidya2020}
N.~Gupta and N.~H. Vaidya, ``Fault-tolerance in distributed optimization: The case of redundancy,'' in \emph{Proceedings of the ACM Symposium on Principles of Distributed Computing (PODC)}, 2020, pp. 365--374.

\bibitem{GuptaDoanVaidya2021}
N.~Gupta, T.~T. Doan, and N.~H. Vaidya, ``Byzantine fault-tolerance in federated local {SGD} under {$2f$}-redundancy,'' \emph{arXiv preprint arXiv:2108.11769}, 2021.

\bibitem{LiuGuptaVaidya2023}
S.~Liu, N.~Gupta, and N.~H. Vaidya, ``Impact of redundancy on resilience in distributed optimization and learning,'' in \emph{Proceedings of the 24th International Conference on Distributed Computing and Networking (ICDCN)}, 2023, pp. 260--269.

\bibitem{DuttaDoan2025OpenRedundancy}
A.~Dutta and T.~T. Doan, ``Distributed optimization in open networks under redundancy,'' in \emph{Proceedings of the 2025 American Control Conference (ACC)}, 2025, pp. 2185--2190.

\bibitem{DuttaDoanReed2023ByzantineFL}
A.~Dutta, T.~T. Doan, and J.~H. Reed, ``Resilient federated learning under byzantine attack in distributed nonconvex optimization with {$2f$}-redundancy,'' in \emph{Proceedings of the 62nd IEEE Conference on Decision and Control (CDC)}, 2023, pp. 1156--1161.

\end{thebibliography}
\end{document}